\documentclass[12pt]{amsart}
\usepackage{graphicx}
\usepackage[active]{srcltx}
\usepackage{amsthm,mathrsfs}
\usepackage{a4wide}
\usepackage[english]{babel}
\usepackage{amsfonts}
\usepackage{amsmath}
\usepackage{amssymb}
\usepackage{dsfont}
\newtheorem{lemma}{Lemma}
\newtheorem{theorem}{Theorem}

\usepackage{float}
\usepackage{url}

\newcommand {\E} {\mathbb{E}}

\DeclareMathOperator {\lcm}{lcm}
\newcommand {\p} {\mathbb{P}}

\newcommand {\N} {\mathbb{N}}

\newcommand {\R} {\mathbb{R}}

\newcommand {\ve} {\varepsilon}

\makeatletter\def\blfootnote{\xdef\@thefnmark{}\@footnotetext}\makeatother
\allowdisplaybreaks
\title[Series of dilated functions and spectral norms of GCD matrices]{\bf Convergence of  series of dilated functions \\ and spectral norms of GCD matrices}

\author[C. Aistleitner]{Christoph Aistleitner}
\address{Department of Mathematics, Graduate School of Science, Kobe University, Kobe 657-8501, Japan}
\email{aistleitner@math.tugraz.at}

\author[I. Berkes]{Istv\'an Berkes}
\address{Institute of Statistics, TU Graz, Kopernikusgasse 24/III, 8010 Graz, Austria}
\email{berkes@tugraz.at}

\author[K. Seip]{Kristian Seip}
\address{Department of Mathematical Sciences, Norwegian University of
Science and Technology (NTNU), NO-7491 Trondheim, Norway}
\email{seip@math.ntnu.no}

\author[M. Weber]{Michel Weber}
\address{IRMA, 10 rue du G\'en\'eral Zimmer,
67084 Strasbourg Cedex, France}
\email{michel.weber@math.unistra.fr}

\thanks{The first author is supported by a Schr\"odinger scholarship
of the Austrian Research Foundation (FWF). The second author is
supported by FWF Grant P 24302-N18 and OTKA grant K 108615. The
third author is supported by the Research Council of Norway grant
227768. This paper was initiated while three of the authors
(Berkes, Seip, Weber) participated in the research program
\emph{Operator Related Function Theory and Time-Frequency
Analysis} at the Centre for Advanced Study at the Norwegian
Academy of Science and Letters in Oslo during 2012--2013.}


\begin{document}

\begin{abstract}
We establish a connection between the $L^2$
norm of sums of dilated functions whose $j$th Fourier coefficients are
 $\mathcal{O}(j^{-\alpha})$ for some $\alpha \in (1/2,1)$,
and the spectral norms of certain greatest common divisor (GCD)
matrices. Utilizing recent bounds for these spectral norms, we
obtain sharp conditions for the convergence in $L^2$ and for the
almost everywhere convergence of series of dilated functions.
\end{abstract}

\date{}
\maketitle

\section{Introduction}

Carleson's theorem~\cite{carleson}  states that the
series
\begin{equation} \label{sincos}
\sum_{k=1}^\infty c_k \sin 2 \pi k x \qquad \textup{and} \qquad \sum_{k=1}^\infty c_k \cos 2 \pi k x
\end{equation}
are convergent for almost every $x$ in $[0,1]$ provided that the sequence of coefficients $(c_k)_{k \geq 1}$ (assumed to be real)  satisfies
\begin{equation} \label{sumck}
\sum_{k=1}^\infty c_k^2 < \infty.
\end{equation}
By orthogonality, condition~\eqref{sumck} is also necessary and
sufficient for  the $L^2$ norm convergence of the two series
in~\eqref{sincos}. A much studied problem is what happens with the convergence in either sense if the functions $\sin 2 \pi x$ and
$\cos 2 \pi x$ are replaced by more general periodic functions.
More precisely, the question is what we can say about the
convergence of  the series
\begin{equation}\label{fkx} 
\sum_{k=1}^\infty c_k f (kx) 
\end{equation}
when $f: \mathbb{R}\to\mathbb{R}$ is a measurable function
satisfying
\begin{equation}\label{f}
f(x+1)=f(x), \qquad \int_0^1 f(x)~dx=0, \qquad \int_0^1
f^2(x)~dx<\infty.
\end{equation}

In general,~\eqref{sumck} will not be a sufficient condition
either for convergence in $L^2$ or for almost everywhere
convergence of~\eqref{fkx}, and the problem is to find alternate
conditions on the coefficients $(c_k)_{k \geq 1}$  when $f$
belongs to a prescribed class of functions. For a survey of
existing results in this direction and recent
results we refer to~\cite{abs, bewe2}. For a recent survey on Carleson's theorem, see~\cite{lacey}.\\

In this paper, we will be interested in the case when $f$ belongs
to the class $C_\alpha$ for  $\alpha>1/2 $, i.e.\ when the Fourier
series of $f$ is of the form
$$
\sum_{j=1}^\infty \left(a_j \sin 2 \pi j x + b_j \cos 2 \pi j x\right)
$$
with
$$
a_j=\mathcal{O}\left(j^{-\alpha}\right), \quad b_j=\mathcal{O}
\left(j^{-\alpha}\right), \qquad \text{as} \ j \to \infty.
$$
The important limiting case $\alpha=1$ is essentially covered by
the results  of~\cite{abs} (see Section~\ref{background} for
details). We will now extend the methods of~\cite{abs} to cover
also the range $1/2<\alpha<1$ and will give sharp conditions for
the $L^2$ convergence and the almost everywhere convergence
of~\eqref{fkx} as well as of the related series
\begin{equation} \label{fnkx}
\sum_{k=1}^\infty c_k f(n_k x),
\end{equation}
where $(n_k)_{k \geq 1}$ is a sequence of distinct positive integers. \\

Problems concerning the convergence of~\eqref{fkx} or~\eqref{fnkx}
can be traced back  to Riemann's \emph{Habilitationsschrift}
(1852). They exhibit profound interrelations between various parts
of analysis and number theory, as illustrated by the following
list of important contributions:  classical formulas of Franel and
Landau connecting the convergence theory of~\eqref{fkx}
and~\eqref{fnkx} to sums of greatest common divisors (GCD sums);
their generalization to the Hurwitz zeta  function due to
Mikol\'as; the work of Koksma, Erd\H os, G\'al, LeVeque, and
others in Diophantine approximation and uniform distribution
theory; the results of Dyer and Harman in the context of the
Duffin--Schaeffer conjecture in metric Diophantine approximation;
upper and lower bounds for GCD sums obtained by the authors of the
present paper; and problems concerning the magnitude of the
largest eigenvalue of GCD matrices, which were studied by Wintner,
by Lindqvist and Seip (in the context of questions about Riesz
bases), and by Hilberdink (in the context of the Riemann zeta
function). Basic work on the convergence and divergence of dilated
series and their relation to lacunary series was done by
Gaposhkin, Nikishin, Philipp, and Kaufman, just to mention a few.\\

In view of this multitude of connections, we have found it
appropriate to give a fairly detailed presentation  of those ideas
and lines of research that are most relevant for our particular
problem.  To this end, following the statement of our three main
theorems in the next section, Section~\ref{background} gives an
extensive survey of relevant background material.
Section~\ref{aux} contains auxiliary results, and the proofs are
given in Section~\ref{clambdap}.

\section{Results} \label{sec2}


Throughout this paper we write $K,\hat{K},K_1,K_2,\dots$ for
appropriate positive  constants, not always the same, which only
depend (at most) on $\alpha$ and $f$. We will use the Vinogradov
symbols ``$\ll$'' and ``$\gg$'' in the same sense. Throughout this
paper, we assume that $(c_k)_{1 \leq k \leq N}$ and $(c_k)_{k \geq
1}$ denote sequences of real numbers and that $(n_k)_{1 \leq k
\leq N}$ and $(n_k)_{k \geq 1}$ denote sequences of distinct
positive integers. For notational convenience, throughout this
paper we will read $\log x$ as $\max~ \{1,\log x\}$; in particular,
this implies that iterated logarithms are defined and non-zero.

\begin{theorem} \label{th1}
Assume that $f \in C_\alpha$ for some $\alpha \in (1/2,1)$. Then
the series~\eqref{fkx} is  convergent in $L^2$ norm and almost
everywhere provided
\begin{equation} \label{th1cond}
\sum_{k=1}^\infty c_k^2 \exp \left( \frac{K (\log k)^{1-\alpha}}{\log \log k} \right) < \infty, \qquad \textrm{where \ $K = 3/(1-\alpha) + 4/\sqrt{2\alpha-1}$}.
\end{equation}
Conversely, for every $\alpha \in (1/2,1)$ there exist a function
$f \in C_\alpha$ and a sequence $(c_k)_{k \geq 1}$  such
that~\eqref{th1cond} holds with $K$ replaced by
$(1-\ve)/(1-\alpha)$ for any $0<\varepsilon<1$, but the
series~\eqref{fkx} is not convergent in $L^2$.
\end{theorem}

\begin{theorem} \label{th2}
Assume that $f \in C_\alpha$ for some $\alpha \in (1/2,1)$. Then
the series~\eqref{fnkx}   is convergent in $L^2$ norm and almost
everywhere if
\begin{equation} \label{th2cond}
\sum_{k=1}^\infty c_k^2  \exp\left( \frac{K (\log k)^{1-\alpha}}{(\log \log k)^{\alpha}} \right) < \infty, \qquad \textrm{where \ $K = 6/(1-\alpha) +
7\left(|\log(2\alpha-1)|^{1/2}+1\right)$}.
\end{equation}
Conversely, for every $\alpha \in (1/2,1)$ there exist a function $f \in C_\alpha$, a sequence $(c_k)_{k \geq 1}$, a sequence $(n_k)_{k \geq 1}$, and a constant $\hat{K}=\hat{K}(\alpha)$ such that~\eqref{th2cond} holds with $K$ replaced by $\hat{K}$, but the series~\eqref{fnkx} is not convergent in $L^2$ norm and is divergent almost everywhere.
\end{theorem}
Theorem~\ref{th1} improves results of Br\'{e}mont~\cite{bremont},
who proved that~\eqref{fkx} is convergent in $L^2$ norm and almost
everywhere provided
$$
\sum_{k=1}^\infty c_k^2 \exp \left( \frac{(1+\ve)(\log k)^{2 - 2 \alpha}}{2(1-\alpha) \log \log k}\right) < \infty \qquad \textrm{for some $\ve>0$.}
$$
Br\'{e}mont also proved that there exists a sequence $(c_k)_{k \geq 1}$ satisfying~\eqref{sumck} such that the series~\eqref{fkx} does not converge in $L^2$ norm and is almost everywhere divergent. \\

As the second part of Theorem~\ref{th2} shows,
condition~\eqref{th2cond} is optimal both for convergence in $L^2$
and almost everywhere convergence, except for the precise value of
the constant, thus providing a nearly complete solution of the
problem of norm convergence and almost everywhere convergence of
series of the form~\eqref{fnkx}. In Theorem~\ref{th1},  we claim
the optimality  of condition~\eqref{th1cond} only for the norm
convergence of~\eqref{fkx}; we do not know whether~\eqref{th1cond}
is optimal also for almost everywhere convergence. However, we
know that, in general, condition~\eqref{sumck} is not sufficient
for the almost everywhere convergence of the series~\eqref{fkx}.
This follows from our proof of the optimality of the convergence
condition in Theorem~\ref{th2} for almost everywhere convergence
of~\eqref{fnkx}. In fact, for the proof of the optimality of
Theorem~\ref{th2} for given $\alpha \in (1/2,1)$ and an
appropriate function $f \in C_\alpha$, we construct sequences
$(c_k)_{k \geq 1}$ and $(n_k)_{k \geq 1}$ such that
condition~\eqref{th2cond} holds for a certain value of $K$, but
the series~\eqref{fnkx} is almost everywhere divergent. The proof
reveals that $n_k$ is of asymptotic order at most $R^{k \log k}$
for some constant $R=R(\alpha)$. Consequently, setting $d_{n_k} =
c_{k}$ when $n=n_k$  and $d_n=0$ otherwise, we see that $\sum_{n=1}^\infty d_n f(n x)$ is divergent almost everywhere, but
\begin{eqnarray*}
\sum_{n=1}^\infty d_n \exp\left( \frac{\hat{K} (\log \log n)^{1-\alpha}}{(\log \log \log n)^{\alpha}} \right) & = & \sum_{k=1}^\infty d_{n_k} \exp\left( \frac{\hat{K} (\log \log n_k)^{1-\alpha}}{(\log \log \log n_k)^{\alpha}} \right) \\
& \leq & \sum_{k=1}^\infty c_k  \exp\left( \frac{K (\log k)^{1-\alpha}}{(\log \log k)^{\alpha}} \right) \quad < \quad \infty
\end{eqnarray*}
for some (sufficiently small) positive constant $\hat{K}$. Hence, in the condition for almost everywhere convergence in Theorem~\ref{th1}, a Weyl factor of order at least
$$
\exp\left( \frac{\hat{K} (\log \log k)^{1-\alpha}}{(\log \log \log k)^{\alpha}} \right)
$$
is necessary. This leaves a rather large gap in comparison to the Weyl factor in~\eqref{th1cond}.\\


As noted, Theorem~\ref{th1} gives an optimal condition for the
problem of $L^2$ convergence of series of the form~\eqref{fkx}.
More precisely, this statement is true as long as one requests the
Weyl multiplier to be a ``simple'', slowly varying function. On
the other hand, the situation is totally different if one allows
the Weyl mutiplier $\psi(k)$ to depend on number-theoretic
properties of $k$ and to be strongly fluctuating as $k$ increases.
In this sense, Theorem~\ref{th1} may be said to conceal the
arithmetical nature of our problem. To state the next result, we
introduce the divisor function
\begin{equation*} \label{sigmasd}
\sigma_{s} (k)=\sum_{d|k} d^{s}.
\end{equation*}
We will prove the following result.

\begin{theorem} \label{th3} Assume that $f \in C_\alpha$ for some $\alpha \in (1/2,1)$. Assume also that
\begin{equation}\label{sigma1}
\sum_{k=1}^\infty c_k^2 \sigma_{1-2\alpha+\ve} (k) <\infty
\end{equation}
for some $\ve>0$. Then~\eqref{fkx} is convergent in $L^2$. On the other hand, for every $\alpha \in (1/2,1)$ and every
$0<\beta<1$ there exist a function $f \in C_\alpha$ and a real sequence $(c_k)_{k \geq 1}$  such that
\begin{equation}\label{sigma2}
\sum_{k=1}^\infty c_k^2 \sigma_{-\alpha} (k)^{\beta} <\infty,
\end{equation}
but~\eqref{fkx} is not convergent in $L^2$.
\end{theorem}

\bigskip
In Berkes and Weber~\cite{bewe} it is proved that
\begin{equation}\label{sigma1b}
\sum_{k=1}^\infty c_k^2 \sigma_{1-2\alpha} (k) (\log k)^2<\infty
\end{equation}
implies the convergence in $L^2$ norm and almost everywhere convergence of~\eqref{fkx}. Despite the similarity of~\eqref{sigma1} and~\eqref{sigma1b}, there is a
crucial difference between the corresponding convergence statements.
Clearly, for every $s>0$ we have
$$
\sum_{k=1}^n \sigma_{-s}(k)=\sum_{k=1}^n \sum_{d|k} d^{-s}=\sum_{d=1}^\infty \left \lfloor \frac{n}{d}\right\rfloor d^{-s}\sim n \sum_{d=1}^\infty d^{-1-s} \qquad \text{as} \ n\to\infty,
$$
showing that the average value of the function $\sigma_{-s}(k)$ is
$\sum_{d=1}^\infty d^{-1-s}<\infty$. This implies that given any
function $\omega (k)\to \infty$, the asymptotic density of the set
$\{k: \sigma_{-s}(k)\le \omega (k) \}$ is $1$  and thus for
$\alpha>1/2$ and sufficiently small $\ve>0$, the Weyl factor
$\sigma_{1-2\alpha+\ve} (k)$ in~\eqref{sigma1} is of order
$\mathcal{O}(\omega(k))$ for ``most'' $k$. Thus, despite the
optimality of the condition
\begin{equation*}
\sum_{k=1}^\infty c_k^2 \exp \left( \frac{K (\log k)^{1-\alpha}}{\log \log k} \right) < \infty
\end{equation*}
in Theorem~\ref{th1}, for most $k$ the much smaller Weyl factor
$\omega (k)$ suffices for the norm convergence of
$\sum_{k=1}^\infty c_k f(kx)$.
This effect will be apparent from the proofs of the divergence
results in Theorems~\ref{th1}--\ref{th3}. The construction of
$(c_k)_{k \geq 1}$ and $(n_k)_{k \geq 1}$ in the 
examples of divergence uses, roughly speaking, the eigenvectors of suitable GCD
matrices belonging to the maximal eigenvalue, which, as is seen
from~\cite{abs} and~\cite{hilber}, are concentrated on indices $k$
with many small prime factors. These are also the indices $k$
where the divisor functions $\sigma_{-s} (k)$ are large: as
Gronwall~\cite{gron} showed,
\begin{equation} \label{grw}
\sigma_{-s} (k)\le \exp \left( \frac{1+o(1)}{1-s} \frac{(\log k)^{1-s}}{\log\log k}\right)
\end{equation}
and $\sigma_{-s} (k)$ reaches the order of magnitude on the right-hand side along the
sequence $k_r=p_1 \cdots p_r$, $r=1, 2, \ldots$, where $(p_r)_{r \ge 1}$ is the sequence of
primes. There is a gap between~\eqref{sigma1} and~\eqref{sigma2}, and the problem of finding the optimal arithmetic
function required for the $L^2$ norm convergence of~\eqref{fkx} remains open.\\

As mentioned in the introduction, the case $\alpha=1$ is
essentially covered by the results of~\cite{abs}. We refer here
to~\cite[Theorem~3]{abs}, concerning the almost everywhere
convergence of~\eqref{fnkx} for functions $f$ of bounded
variation. The only property used in the proof of that result is
that a function of bounded variation belongs to $C_1$. It
therefore follows from ~\cite[Theorem~3]{abs} that~\eqref{fnkx} is almost everywhere convergent when $f\in C_1$ provided
\begin{equation} \label{ckgamma}
\sum_{k=1}^\infty c_k^2 (\log \log k)^\gamma < \infty
\end{equation}
for some $\gamma > 4$ (under the additional assumption that
$(n_k)_{k \geq 1}$ is strictly increasing).
Moreover, it was proved in~\cite[Theorem~7]{abs} that this statement becomes
false for $\gamma<2$. Since the series~\eqref{fkx} is a special case of~\eqref{fnkx}, the series~\eqref{fkx} is also almost everywhere convergent for all $f\in C_1$ if~\eqref{ckgamma} holds for some $\gamma>4$. Concerning $L^2$ convergence, using~\cite[Lemma~4]{abs} it can be shown that the series~\eqref{fnkx} is convergent in $L^2$ norm for all $f \in C_1$ provided~\eqref{ckgamma} holds for some $\gamma>4$, and by the results in~\cite{gal} this statement becomes false for $\gamma<2$. Moreover, using the results from~\cite{hilber} it is possible to show that that the series~\eqref{fkx} is convergent in $L^2$ norm for all $f \in C_1$ provided~\eqref{ckgamma} holds for some $\gamma>2$, and this statement also becomes false for $\gamma<2$. Thus the problem of $L^2$ convergence and almost everywhere convergence of the series~\eqref{fkx} and~\eqref{fnkx} is solved, up to powers of $(\log \log k)$ in the extra convergence conditions. The problem of norm and almost everywhere convergence of~\eqref{fkx} 
when~\eqref{f} is our only assumption on $f$, is considerably harder. The reason for the
difficulties is that while for $f \in C_\alpha$ we have
\begin{equation} \label{covar}
\left| \int_0^1 f(kx)f(\ell x) ~dx \right| \le K \frac{(\gcd(k,
\ell))^{2\alpha}}{(k\ell)^\alpha}, \qquad k, \ell \ge 1
\end{equation}
for some constant $K>0$, for general $f$ satisfying~\eqref{f} the
integral in \eqref{covar} depends on $k, \ell$ and the Fourier
coefficients of $f$ in a rather complicated way and the arithmetic
machinery involving GCD sums and eigenvalues of GCD matrices used
in the proof of our theorems breaks down. Assuming that the
complex Fourier coefficients $a_j$ of $f$ satisfy  $|a_j|\le
\phi(j)$, where the positive function $\phi$ has the homogeneity
property $|\phi(jk)|\ll k^{-\gamma} \phi(j)$ for some $\gamma>0$,
much of what is developed in the present paper will carry over to
this situation. Estimates as those found in~\cite{BonS} could
then, for instance, be used to obtain  fairly sharp
analogues of Theorem~\ref{th1} and Theorem~\ref{th2} for the considered function
classes. 
\\

 In case of arithmetic criteria like in Theorem 3, Berkes
 and Weber~\cite{bewe3} proved that if $f$ satisfies~\eqref{f} with complex Fourier coefficients $a_j$, then the series~\eqref{fkx} converges
almost everywhere provided
\begin{equation}\label{ccond}
\sum_{k=1}^\infty c_k^2 \psi(k) (\log k)^2<\infty,
\end{equation}
where the arithmetic function $\psi$ is defined by
\begin{equation}\label{conds}
\psi(k)=\sum_{d|k} (dg(d)+G(d)) \quad  \text{where} \quad  g(r)=
\sum_{j=1}^\infty |a_{jr}|^2, \quad G(r)=\sum_{j\le 2r}
g(j).
\end{equation} For example, if $|a_j| \le  K j^{-1/2} (\log
j)^{-\gamma}$, $\gamma
>1/2$, then $\psi (k)$ reduces to
\begin{equation}\label{conds2}
\psi (k) = \sum_{d|k} (\log d)^{-(2\gamma-1)}.
\end{equation}
Note that the arithmetic function $\psi$ in~(\ref{conds2}) is
larger than the one in~(\ref{sigma1}), which is of course to be
expected. Note also that if $j^{-\gamma} |a_j|$ is non-increasing
for some $\gamma>0$, then in~(\ref{ccond}) we can choose
$$\psi(k)=d(k)=\sum_{d|k} 1.$$
The same criterion holds if $f$ satisfies a H\"older continuity condition, see~\cite{bewe,weber}. These
remarks show again the strong arithmetic character of our
convergence problem. In~\cite{bewe3} it is also shown that except
the factor $(\log k)^2$, condition~(\ref{ccond}) is optimal.
However, just like in Theorem 3, the arithmetic criterion~(\ref{ccond}) is not as sharp as those in Theorems~\ref{th1} and~\ref{th2}. \\

Note that if~(\ref{fkx}) converges almost everywhere for $c_k=1/k$, then by
the Kronecker lemma we have
\begin{equation}\label{khconj}
\lim_{N\to\infty} \frac {1}{N} \sum_{k=1}^N f(kx)=0 \qquad
\textup{a.e.}
\end{equation}
and thus the almost everywhere convergence problem of~(\ref{fkx}) under
(\ref{f}) is closely connected with the classical problem of the
convergence of averages in~(\ref{khconj}).  Khinchin~\cite{khin}
conjectured that under~(\ref{f}) (even without the third
condition) the convergence relation~(\ref{khconj}) holds. This
conjecture was disproved nearly 50 years later by a famous
counterexample of Marstrand~\cite{mars}. In the positive
direction, Koksma~\cite{koks53} proved that~(\ref{khconj}) holds
provided the complex Fourier coefficients $a_j$ of $f$ satisfy
\begin{equation*}\label{kokcond}
\sum_{j=1}^\infty |a_j|^2 \sigma_{-1} (j)<\infty.
\end{equation*}
Bourgain~\cite{bo} gave a new, much simplified counterexample to
Khinchin's conjecture and claim\-ed, without proof, that Koksma's
criterion is essentially optimal. This claim was proved recently
by Berkes and Weber~\cite{bewe3}.  Thus while the almost everywhere
convergence problem for~(\ref{fkx}) under~(\ref{f}) remains open,
the closely related problem of almost everywhere convergence of averages
(\ref{khconj}) is essentially settled.

\section{The role of GCD matrices and certain extremal functions in $C_\alpha$}\label{background}

We will now review the key ideas used in both~\cite{abs} and the
present paper. We begin by introducing the special functions
$f_\alpha(x)$ and $\bar{f}_\alpha(x)$ in $C_\alpha$ defined by
\begin{equation} \label{flambda}
f_\alpha (x) = \sum_{j=1}^\infty \frac{\sin 2 \pi j x}{j^\alpha} \qquad \textrm{and} \qquad \bar{f}_\alpha(x) = \sum_{j=1}^\infty \frac{\cos 2 \pi j x}{j^\alpha}.
\end{equation}
Informally speaking, these functions are extremal in $C_\alpha$ in the sense that their Fourier coefficients are of maximal size.
Furthermore, all Fourier coefficients are positive, which makes it relatively easy to obtain lower bounds for $L^2$ norms of sums of dilated functions. \\

When $\alpha=1$, the first series in~\eqref{flambda} is the Fourier series of the function
$$
f_1(x) = \pi \left(1/2 - \{x\} \right),
$$
where $\{ \cdot \}$ denotes fractional part. This means that, up to multiplication by a constant, $f_1$ is the first Bernoulli polynomial on $[0,1]$, extended with period one. Convergence problems for ~\eqref{fkx} and~\eqref{fnkx} have been investigated extensively for $f=f_1$, starting probably with Riemann's \emph{Habilitationsschrift} of 1852. Such series have been called \emph{Davenport series} in honor of Harold Davenport, who was the first to study them in this general form~\cite{davenport}. See~\cite{jaffard} for a survey on the history of the subject and several results on the convergence problem for series involving this function. Convergence problems for Davenport series have an interesting connection with fractal geometry, see for example~\cite{jaffard2}.\\

The convergence problem for series involving the function $f_1$ is connected with sums involving greatest common divisors through the formula
\begin{equation} \label{franel}
\int_0^1 \left( \{kx\}-1/2\right) (\{\ell x\}-1/2) dx = \frac{1}{12} \frac{(\gcd(k,\ell))^2}{k \ell}
\end{equation}
for positive integers $k,\ell$, which was first stated by Franel and formally proved by Landau in 1924. Consequently we have
\begin{equation} \label{landau1}
\int_0^1 \left( \sum_{k=1}^N c_k f_1(n_k x) \right)^2 dx = \frac{\pi^2}{12} \sum_{k,\ell=1}^N c_k c_\ell \frac{(\gcd(n_k,n_\ell))^2}{n_k n_\ell}.
\end{equation}
But much more is true since the Fourier coefficients of $f_1$ are positive and maximal: By an observation of Koksma~\cite{koks} we have
\begin{equation} \label{landau2}
\int_0^1 \left( \sum_{k=1}^N c_k f(n_k x) \right)^2 dx \ll \sum_{k,\ell=1}^N |c_k c_\ell | \frac{(\gcd(n_k,n_\ell))^2}{n_k n_\ell}
\end{equation}
for every function $f$ in $C_1$. \\

The relation between $L^2$ norms of sums of dilated functions and
sums involving greatest common  divisors extends to the classes
$C_\alpha$ for $1/2<\alpha < 1$. This was first observed by
Mikol\'as~\cite{mikolas1957}, who proved that for the Hurwitz zeta
function $\zeta(1-\alpha,\cdot)$ we have
\begin{equation} \label{mikol}
\int_0^1 \zeta(1-\alpha,\{k x\}) \zeta(1-\alpha,\{\ell x\}) ~dx = 2 \Gamma(\alpha)^2 \frac{\zeta(2 \alpha)}{(2\pi)^{2\alpha}} \frac{\left(\gcd(k,\ell)\right)^{2\alpha}}{(k \ell)^\alpha}
\end{equation}
for positive integers $k,\ell$ and for $\alpha>1/2$. Hurwitz's formula states that for $\alpha>1$ and $x \in [0,1]$ we have
\begin{equation*}
\zeta(1-\alpha,x) = \frac{\Gamma(\alpha)}{(2 \pi)^\alpha} \left( e^{-\pi i \alpha/2} \left( \sum_{j=1}^\infty \frac{e^{2 \pi i j x}}{j^\alpha} \right) + e^{\pi i \alpha/2} \left( \sum_{j=1}^\infty \frac{e^{-2 \pi i j x}}{j^\alpha} \right) \right)
\end{equation*}
(see for example~\cite{knopprobins} for a simple proof), which implies that
\begin{equation} \label{mikol2}
\zeta(1-\alpha,x) = \frac{2 \Gamma(\alpha)}{(2 \pi)^\alpha} \left( \cos(\pi \alpha/2) \left( \sum_{j=1}^\infty \frac{\cos 2 \pi j x}{j^\alpha} \right) + \sin (\pi \alpha/2) \left( \sum_{j=1}^\infty \frac{\sin 2 \pi j x}{j^\alpha}  \right) \right).
\end{equation}
Thus $\zeta(1-\alpha,x)$ is a function whose Fourier coefficients
are precisely of asymptotic order $j^{-\alpha}$, and in particular
$\zeta(1-\alpha,x) \in C_\alpha$.  As Mikol\'as
showed, the formula~\eqref{mikol2} continues to hold for $\alpha > 1/2$ and $0< x < 1$, which leads  to~\eqref{mikol} by the orthogonality of
the trigonometric system. By the same argument as for the case
$\alpha=1$, we get that
\begin{equation} \label{l2norm}
\int_0^1 \left( \sum_{k=1}^N c_k f(n_k x) \right)^2 dx \ll \sum_{k,\ell=1}^N |c_k c_\ell | \frac{\left(\gcd(n_k,n_\ell)\right)^{2\alpha}}{(n_k n_\ell)^\alpha}
\end{equation}
for every function $f$ in $C_\alpha$ (see Lemma~\ref{lemmaalpha}
below). For the special function $f_\alpha(x)$
from~\eqref{flambda} we get
\begin{equation} \label{l2normlam}
\int_0^1 \left( \sum_{k=1}^N c_k f_\alpha(n_k x) \right)^2 dx = \frac{\zeta(2 \alpha)}{2} \sum_{k,\ell=1}^N c_k c_\ell \frac{\left(\gcd(n_k,n_\ell)\right)^{2\alpha}}{(n_k n_\ell)^\alpha},
\end{equation}
as will also be established in Lemma~\ref{lemmaalpha} below. \\

Our two estimates~\eqref{landau2} and~\eqref{l2norm}, as well as the two identities ~\eqref{landau1} and~\eqref{l2normlam}, show that to understand the convergence of~\eqref{fkx} and~\eqref{fnkx} for $f$ in $C_\alpha$ it is important to have good upper and lower bounds for sums of the form
\begin{equation} \label{sumform}
\sum_{k,\ell=1}^N c_k c_\ell \frac{\left(\gcd(k,\ell)\right)^{2\alpha}}{(k \ell)^\alpha}\qquad \textrm{and} \qquad \sum_{k,\ell=1}^N c_k c_\ell \frac{\left(\gcd(n_k,n_\ell)\right)^{2\alpha}}{(n_k n_\ell)^\alpha}.
\end{equation}
Now let $G_N^{(\alpha)}$ be the $N \times N$ matrix with entries $g_{k \ell}$ given by
\begin{equation} \label{wint}
g_{k \ell} = \frac{(\gcd(k,\ell))^{2 \alpha}}{(k \ell)^\alpha}
\end{equation}
and $H_N^{(\alpha)}$ the $N \times N$ matrix with entries $h_{k \ell}$ of the form
$$
h_{k \ell} = \frac{(\gcd(n_k,n_\ell))^{2 \alpha}}{(n_k n_\ell)^\alpha}.
$$
It is a well-known fact that both of these matrices are positive definite (see e.g.~\cite{linds}). Thus for the largest eigenvalue $\Lambda\left(G_N^{(\alpha)}\right)$ of $G_N^{(\alpha)}$ we have
\begin{equation} \label{eigenv1}
\Lambda\left(G_N^{(\alpha)}\right) = \max_{\substack{c_1,\dots,c_N:\\~c_1^2+\dots+c_N^2=1}} ~\sum_{k,\ell=1}^N c_k c_\ell \frac{\left(\gcd(k,\ell)\right)^{2\alpha}}{(k \ell)^\alpha},
\end{equation}
and for the largest eigenvalue $\Lambda\left(H_N^{(\alpha)}\right)$ of $H_N^{(\alpha)}$ we have
\begin{equation} \label{eigenv2}
\Lambda\left(H_N^{(\alpha)}\right) = \max_{\substack{c_1,\dots,c_N:\\~c_1^2+\dots+c_N^2=1}} ~\sum_{k,\ell=1}^N c_k c_\ell \frac{\left(\gcd(n_k,n_\ell)\right)^{2\alpha}}{(n_k n_\ell)^\alpha}.
\end{equation}
Consequently, by~\eqref{l2norm} and~\eqref{l2normlam}, the problem of finding upper and lower bounds for the largest eigenvalue (or the square-root of the spectral norm) of $G_N^{(\alpha)}$ and $H_N^{(\alpha)}$ is precisely  the same as that of finding general upper bounds for respectively
\begin{equation}\label{squint}
\int_0^1 \left( \sum_{k=1}^N c_k f(k x) \right)^2 dx \qquad \text{and} \qquad \int_0^1 \left( \sum_{k=1}^N c_k f(n_k x) \right)^2 dx
\end{equation}
when $f$ is  in $ C_\alpha$, and of finding lower bounds for these integrals in the special case when $f=f_\alpha$.\\

The problem of calculating the largest eigenvalue $\Lambda \left(G_N^{(\alpha)}\right)$ of $G_N^{(\alpha)}$, and accordingly the problem of estimating the integral on the left-hand side of~\eqref{squint}, was solved by Hilberdink~\cite{hilber}, who proved that
\begin{equation} \label{hilb1}
\Lambda \left(G_N^{(\alpha)}\right) = \frac{1}{\zeta(2)} \left( e^\gamma \log \log N + \mathcal{O}(1) \right)^2 \qquad \textrm{for} \qquad \alpha = 1
\end{equation}
and
\begin{equation} \label{hilb2}
\Lambda \left(G_N^{(\alpha)}\right) \ll \exp \left(K \frac{(\log N)^{1-\alpha}}{\log \log N}\right) \qquad \textrm{for} \qquad \frac{1}{2} < \alpha < 1.
\end{equation}
In~\eqref{hilb2} the constants $K$ depends on $\alpha$, and~\eqref{hilb2} is optimal except for the precise value of $K$. For $H_N^{(\alpha)}$, in Lemma~4 and Theorem~5 of~\cite{abs} it was shown that
 \begin{equation} \label{cknorms1}
\Lambda \left(H_N^{(\alpha)}\right) \ll (\log \log N)^4 \qquad \textrm{for} \qquad \alpha=1
\end{equation}
and
\begin{equation} \label{cknorms2}
\Lambda \left(H_N^{(\alpha)}\right) \ll \exp \left(K \frac{(\log N)^{1-\alpha}}{(\log \log N)^\alpha} \right) \qquad \textrm{for} \qquad \frac{1}{2} < \alpha < 1,
\end{equation}
where the constant $K$ depends on $\alpha$. Here~\eqref{cknorms2} is optimal except for the precise value of the constant $K$, but it remains a profound problem to decide whether the exponent $4$ of $\log\log N$ on the right-hand side of~\eqref{cknorms1} is optimal. By a classical theorem of G\'al~\cite{gal}, it is known that this exponent can not be smaller than 2.\footnote{Note added in proof: This problem was solved recently by Lewko and Radziwi{\l}{\l}~\cite{lewko}, which also means an improvement of the results for the convergence problem for $f \in C_1$, mentioned in Section~\ref{sec2}. The key point of the proof in~\cite{lewko} is to show that the GCD sum on the right-hand side of \eqref{sumform} is essentially dominated by the square of the maximum of a certain random model of the Riemann zeta function. This argument intensifies the relationship between GCD sums and problems concerning the maximum of the Riemann zeta function, which was already anticipated in~\cite{abs} and~\cite{hilber}. See also~\
cite{a2}.}\\

As noted above, the results~\eqref{hilb1}--\eqref{cknorms2} imply corresponding upper bounds for the integrals in~\eqref{squint} when $f \in C_\alpha$, and the optimality of~\eqref{hilb1},~\eqref{hilb2} and~\eqref{cknorms2} implies corresponding lower bounds for the integrals in~\eqref{squint} in the special case when $f=f_\alpha$; this is the reason why the exponential factor from~\eqref{hilb2} appears and Theorem~\ref{th1}, and that from~\eqref{cknorms2} appears in Theorem~\ref{th2}. When comparing the bounds for the largest eigenvalues of $G_N^{(\alpha)}$ and $H_N^{(\alpha)}$, respectively, we note that in the case $\alpha=1$ there is an additional factor $(\log \log N)^2$ in~\eqref{cknorms1} as compared with~\eqref{hilb1}. As mentioned above, this extra factor possibly can be avoided since we do not know whether~\eqref{cknorms1} is optimal.\footnote{Cf. Footnote 1.} In the case $1/2<\alpha<1$ there is a difference between the denominator in the exponential terms in~\eqref{hilb2} and~\eqref{cknorms2}, 
respectively, which is $\log \log N$ in the one case and $(\log \log N)^\alpha$ in the other case. Since both results are optimal, this shows that there really is a significant difference between the spectral norms of $G_N^{(\alpha)}$ and of $H_N^{(\alpha)}$, and accordingly also a difference between the convergence problems for~\eqref{fkx} and~\eqref{fnkx}. In~\cite{hilber}, a connection is established between the spectral norm of $G_N^{(\alpha)}$ and the maximal order of magnitude of the Riemann zeta-function along vertical lines, using Soundararajan's ``resonance method'' from~\cite{sound}. However, Hilberdink's results cannot reach the stronger lower bounds of Montgomery~\cite{mont}, which in turn bear a striking resemblance to the bounds for the spectral norm of $H_N^{(\alpha)}$ in~\cite{abs}.\footnote{Again, cf. Footnote 1, as well as~\cite{a2}.}
\\

We close this section by making an observation on our extremal functions $f_\alpha$ and $\bar{f}_\alpha$ in~\eqref{flambda} that will be needed in the sequel. We note first that they are, up to normalization,  the even and odd parts of the Hurwitz zeta function. In fact, from the Fourier series representation in~\eqref{mikol2} it is easily seen that
\begin{eqnarray}
\frac{\zeta(1-\alpha,x) - \zeta(1-\alpha,1-x)}{2} & = & \frac{2 \Gamma(\alpha)}{(2 \pi)^\alpha} \sin(\pi \alpha/2) f_\alpha(x) \label{hurw1}
\end{eqnarray}
and
\begin{eqnarray}
\frac{\zeta(1-\alpha,x) + \zeta(1-\alpha,1-x)}{2} & = & \frac{2 \Gamma(\alpha)}{(2 \pi)^\alpha} \cos(\pi \alpha/2) \bar{f}_\alpha(x). \label{hurw2}
\end{eqnarray}
These representations can be used to describe the rate with which $f_\alpha(x)$ and $\bar{f}_\alpha(x)$ tend to infinity as $x \to 0$. Mikol\'as proved that for fixed $\alpha \in (1/2,1)$ we have
$$
\lim_{x \to 0+} x^{1-\alpha} \zeta(1-\alpha,x) = 1
$$
(this is equation (12) in~\cite{mikolas1957}). Consequently, since $\lim_{x \to 0+} \zeta(1-\alpha,1-x)=\zeta(1-\alpha,1)=\zeta(1-\alpha)$ is a constant, we have
$$
\lim_{x \to 0+} x^{1-\alpha} f_\alpha(x) = \frac{(2 \pi)^\alpha}{\Gamma(\alpha) \sin(\pi \alpha/2)}.
$$
In particular this implies that
\begin{equation} \label{falphalp}
f_\alpha \in L^p(0,1) \qquad \textrm{for} \qquad p < \frac{1}{1-\alpha},
\end{equation}
which will be a crucial ingredient in the proof of the necessary condition for almost everywhere convergence of~\eqref{fnkx}. More precisely,~\eqref{falphalp} implies that for any $\alpha \in (1/2,1)$ the function $f_\alpha$ is in $L^{2+\delta}$ for some $\delta=\delta(\alpha)>0$, which will allow us to apply Lyapunov's central limit theorem (which requires the existence of an absolute moment of order $2+\delta$ for some $\delta>0$). Similar results hold if $f_\alpha$ is replaced by $\bar{f}_\alpha$.\\

\section{Auxiliary results} \label{aux}

In the sequel, we use the notation $\| \cdot \|$ for the $L^2(0,1)$ norm. Throughout the rest of this paper, we will always assume that $\alpha \in (1/2,1)$.

\begin{lemma} \label{lemmaalpha}
Assume that $f \in C_\alpha$. Then
$$
\int_0^1 \left( \sum_{k=1}^N c_k f(n_k x) \right)^2 dx \ll \sum_{k=1}^N |c_k c_\ell | \frac{(\gcd(n_k,n_\ell))^{2 \alpha}}{(n_k n_\ell)^\alpha}.
$$
For the particular function $f_\alpha$ from~\eqref{flambda} we have
\begin{equation} \label{lemmaalpha2}
\int_0^1 \left( \sum_{k=1}^N c_k f_{\alpha}(n_k x) \right)^2 dx = \frac{\zeta(2 \alpha)}{2} \sum_{k,\ell=1}^N c_k c_\ell \frac{\left(\gcd(n_k,n_\ell)\right)^{2\alpha}}{(n_k n_\ell)^\alpha}.
\end{equation}
\end{lemma}

Note that as a special case of Lemma~\ref{lemmaalpha} we have
\begin{equation}\label{normest}
\int_0^1 \left( \sum_{k=1}^N c_k f(k x) \right)^2 dx \ll \sum_{k, \ell=1}^N |c_k c_\ell | \frac{(\gcd(k,\ell))^{2 \alpha}}{(k \ell)^\alpha}.
\end{equation}

\begin{proof}[Proof of Lemma~\ref{lemmaalpha}]
The argument needed for the proof of Lemma~\ref{lemmaalpha} is a simple generalization of the arguments leading to~\eqref{landau1} and~\eqref{landau2}, respectively. We write
$$
f(x) \sim \sum_{j=1}^\infty a_j \sin 2 \pi j x,
$$
assuming, to shorten formulas, that $f$ is an odd function; the proof in the general case is exactly the same. Then, by the orthogonality of the trigonometric system, for arbitrary positive integers $m,n$ we have
\begin{eqnarray}
\int_0^1 f(m x) f(n x) ~dx & = & \frac{1}{2} \sum_{j_1,j_2=1}^\infty a_{j_1} a_{j_2} \mathds{1}(j_1 m = j_2 n) \label{j1j2} \\
& = & \frac{1}{2} \sum_{j=1}^\infty a_{j m/\gcd(m,n)} a_{j n/\gcd(m,n)} \label{formu} \\
& \ll & \sum_{j=1}^\infty \left(\frac{\gcd(m,n)}{jm} \right)^\alpha \left(\frac{\gcd(m,n)}{jn} \right)^\alpha \nonumber\\
& \ll & \left(\frac{(\gcd(m,n))^2}{mn} \right)^\alpha. \nonumber
\end{eqnarray}
In~\eqref{j1j2}, we used the fact that $j_1 m = j_2 n$ holds if and only if $j_1 = j n/\gcd(m,n)$ and $j_2 = j m/\gcd(m,n)$ for some positive integer $j$. Applying this inequality for all pairs $(n_k,n_\ell)$ gives the first part of the lemma.\\

In the case $f = f_\alpha$ we have $a_j = j^{-\alpha},~j \geq 1$. Inserting this into~\eqref{formu} we get
$$
\int_0^1 f(m x) f(n x) ~dx = \frac{1}{2} \sum_{j=1}^\infty \left(\frac{\gcd(m,n)}{jm} \right)^\alpha \left(\frac{\gcd(m,n)}{jn} \right)^\alpha = \frac{\zeta(2\alpha)}{2}  \left(\frac{(\gcd(m,n))^2}{mn} \right)^\alpha.
$$
Again we obtain the desired result by summing over all pairs $(n_k,n_\ell)$.
\end{proof}

\begin{lemma} \label{lemma1}
Assume that $f \in C_\alpha$. There exist constants $K_1,K_2$ such that
$$
\left\| \sum_{k=1}^N c_k f(kx) \right\|^2 \ll \exp \left( \frac{K_1 (\log N)^{1-\alpha}}{\log \log N} \right) \sum_{k=1}^N c_k^2
$$
and
$$
\left\| \sum_{k=1}^N c_k f(n_k x) \right\|^2 \ll \exp \left( \frac{K_2 (\log N)^{1-\alpha}}{(\log \log N)^\alpha} \right) \sum_{k=1}^N c_k^2.
$$
We can choose $K_1,K_2$ such that
$$
K_1<3/(1-\alpha)+4/\sqrt{2\alpha-1} \qquad \textrm{and} \qquad K_2< 6/(1-\alpha)+7\left(|\log(2\alpha-1)|^{1/2}+1\right).
$$
\end{lemma}

By Lemma~\ref{lemmaalpha} and~\eqref{eigenv1} and~\eqref{eigenv2}, the estimates in Lemma~\ref{lemma1} follow from corresponding upper bounds for the largest eigenvalues of the matrices $G_N^{(\alpha)}$ and $H_N^{(\alpha)}$, respectively, which were already stated in~\eqref{hilb2} and~\eqref{cknorms2}. The given value for $K_1$ is a coarse estimate for that stated in a more precise form in the proof of~\cite[Theorem 2.3]{hilber} and at the end of~\cite[Section 3]{hilber}; the value for $K_2$ is obtained by using the method of the recent paper~\cite{BonS}, which improves in a significant way the  arguments from~\cite{abs}.\\

Using the same method as in the proof of the Rademacher--Menshov inequality, we easily obtain the following lemma, which is a maximal version of Lemma~\ref{lemma1}. Note that the proof of the Rademacher--Menshov inequality gives an additional logarithmic factor, which however in our case can be included in the exponential term if we slightly increase the value of the constants.

\begin{lemma} \label{lemma1b}
Assume that $f \in C_\alpha$. Then there exist constants $K_1,K_2$ such that
$$
\left\| \max_{1 \leq M \leq N} \left| \sum_{k=1}^M c_k f(kx) \right| \right\|^2 \ll \exp \left( \frac{K_1 (\log N)^{1-\alpha}}{\log \log N} \right) \sum_{k=1}^N c_k^2
$$
and
$$
\left\| \max_{1 \leq M \leq N} \left| \sum_{k=1}^M c_k f(n_kx) \right|  \right\|^2 \ll \exp \left( \frac{K_2 (\log N)^{1-\alpha}}{(\log \log N)^\alpha} \right) \sum_{k=1}^N c_k^2.
$$
We can choose $K_1,K_2$ such that
$$
K_1<3/(1-\alpha)+4/\sqrt{2\alpha-1} \qquad \textrm{and} \qquad K_2< 6/(1-\alpha)+7\left(|\log(2\alpha-1)|^{1/2}+1\right).
$$
\end{lemma}

\begin{lemma}[{\cite[Lemma~6]{a1}}] \label{lemma4}
Assume that for every given $\ve>0$ there exists an $M_0(\ve)$ such that
\begin{equation} \label{*}
\left\| \sup_{M > M_0} \left| \sum_{k=M_0+1}^M c_k f(k x) \right| \right\| \leq \ve.
\end{equation}
Then
$$
\sum_{k=1}^\infty c_k f(k x)
$$
is almost everywhere convergent.
\end{lemma}

For the formulation of the following lemma we note that the unit interval, equipped with Borel sets and Lebesgue measure, is a probability space. Throughout the rest of this paper, we will use the symbols $\p$ and $\E$ with respect to this probability space. Furthermore, throughout the rest of this paper we write $\log_2$ for the dyadic logarithm, and we will read $\log_2 x$ as $\max~ \{1,\log_2 x\}$.\\

The following lemma is a variant of~\cite[Lemma~5]{abs}.

\begin{lemma} \label{lemmaber}
For given $\alpha \in (1/2,1)$, set $\eta= 12/(2 \alpha - 1)$ and let $1\le S_1<T_1<S_2<T_2<\dots$ be integers such that
$$
S_{i+1} \ge T_i + \eta \log_2 i.
$$
Furthermore, let $\Delta_1, \Delta_2,\dots$ be sets of integers such
that $\Delta_i \subset [2^{S_i}, 2^{T_i}]$ and each element of $\Delta_i$ is
divisible by $2^{S_i}$. For $i \geq 1$ and $x \in (0,1)$ set
$$
X_i= X_i(x):=\sum_{k \in \Delta_i} f_\alpha(k x).
$$
Then there exist \emph{independent} random variables $Y_1,Y_2,\dots$ on the
probability space $((0,1),\mathcal B, {\mathbb P})$ such that $\mathbb{E}Y_i=0$ and
$$
\|X_i-Y_i\|\ll i^{-2} \cdot \# \Delta_i.
$$
\end{lemma}

For the proof of Lemma~\ref{lemmaber}, we need the following lemma, which is~\cite[Lemma~3.1]{berk2}. Here, given an integrable function $g(x)$ on $[0,1]$ and an arbitrary integer $m$, we write $[g]_m$ for the function which takes the constant value
$$
m \int_{k/m}^{(k+1)/m} g(x) ~dx
$$
in the intervals $[k/m,(k+1)/m)$, for $k=0, \dots, m-1$.

\begin{lemma}[{\cite[Lemma~3.1]{berk2}}] \label{lemmaber2}
Assume that $f \in C_\alpha$. Let $k \geq 1$ be a positive integer, and write $g(x)=f(k x)$. Then for any integer $m \geq k$ we have
$$
\left\| g - [g]_m \right\| \ll \left( \frac{k}{m} \right)^{(2 \alpha -1)/6}.
$$
\end{lemma}

\begin{proof}[Proof of Lemma~\ref{lemmaber}:]
Let $\mathcal{F}_i$ denote the $\sigma$-field generated by the dyadic intervals
\begin{equation} \label{ujdef}
U_j := \left[j 2^{-S_{i+1}},(j+1)2^{-S_{i+1}}\right), \qquad 0 \leq j < 2^{S_{i+1}},
\end{equation}
and set
$$
\xi_k = \xi_k(\cdot) = \E \left( f_\alpha (k \cdot) | \mathcal{F}_i \right), \qquad k \in \Delta_i,
$$
and
$$
Y_i = Y_i(x) = \sum_{k \in \Delta_i} \xi_k(x).
$$
Then we clearly have $\E \xi_k = 0$, which implies $\E Y_i = 0$. By Lemma~\ref{lemmaber2} and~\eqref{*} for every $k \in \Delta_i$ we have 
$$
\left\| \xi_k(\cdot) - f_\alpha (k \cdot) \right\| \ll \left( \frac{k}{2^{S_{i+1}}}\right)^{(2\alpha-1)/6} \ll \left( \frac{2^{T_i}}{2^{T_i + \eta \log_2 i}}\right)^{(2\alpha-1)/6} \ll i^{-\eta (2 \alpha-1) /6} \ll i^{-2},
$$
which implies that
$$
\left\| X_i - Y_i \right\| \ll i^{-2} \cdot \# \Delta_i.
$$
Since by assumption every $k \in \Delta_{i+1}$ is a multiple of $2^{S_{i+1}}$, each interval $U_j$ in~\eqref{ujdef} is a period interval of $f_\alpha(kx)$ for all $k \in \Delta_{i+1}$, and consequently also for $\xi_k$ for all $k\in \Delta_{i+1}$. Consequently $Y_{i+1}$ is independent of the $\sigma$-field $\mathcal{F}_i$. Since $\mathcal{F}_1 \subset \mathcal{F}_2 \subset \dots$ and since $Y_i$ is $\mathcal{F}_i$-measurable, the random variables $Y_1, Y_2, \dots$ are independent.
\end{proof}

The following lemma is a simple consequence of~\cite[Proposition~3.1]{hilber}, from which it can be deduced in the same way as relation (3.2) of~\cite{hilber}.

\begin{lemma} \label{lemmahil}
We have
\begin{equation*}\label{h32}
\sum_{k, \ell=1}^{N} c_k c_\ell \frac{(\text{gcd} (k, \ell))^{2\alpha}}{(k\ell)^\alpha} \le \sum_{k=1}^{{N^2}} b_k^2,
\end{equation*}
where $b_k$ are defined by
\begin{equation*}
b_k=\frac{1}{k^\alpha} \sum_{d|k} d^\alpha |c_d|.
\end{equation*}
\end{lemma}

\section{Proofs} \label{clambdap}

\begin{proof}[Proof of the convergence part of Theorem~\ref{th1}:]
Throughout this proof, we will write $K_1$ for the constant in the statement of Theorem~\ref{th1}, and $K_2$ for the constant in the statement of the first part of Lemma~\ref{lemma1}. Note that we can assume that $K_1 > K_2$. Relation~\eqref{th1cond} implies that
$$
\sum_{k=e^m+1}^{e^{m+1}} c_k^2 \exp \left( \frac{K_1 (\log k)^{1-\alpha}}{\log \log k}\right) \ll 1 \qquad \textrm{for $m \geq 1$},
$$
which also implies that
$$
\sum_{k=e^m+1}^{e^{m+1}} c_k^2 \ll \exp \left( \frac{-K_1 m^{1-\alpha}}{\log m}\right) \qquad \textrm{for $m \geq 1$}.
$$
Consequently by Lemma~\ref{lemma1} we have, for any $M,N$ satisfying $e^m < M < N < e^{m+1}$,
\begin{eqnarray}
\left\| \sum_{k=M}^N c_k f(kx) \right\|^2 & \ll & \exp \left( \frac{K_2 (m+1)^{1-\alpha}}{\log (m+1)}\right) \exp \left( \frac{-K_1 m^{1-\alpha}}{\log m}\right) \nonumber\\
& \ll & \exp \left( \frac{-\ve m^{1-\alpha}}{\log m}\right) \label{m0equ}
\end{eqnarray}
for some $\ve>0$, since $K_1>K_2$. For given $M<N$, let $\hat{m}$ denote the integer for which $M \in \left(e^{\hat{m}},e^{\hat{m}+1}\right]$, and $\hat{n}$ the integer for which $N \in \left(e^{\hat{n}},e^{\hat{n}+1}\right]$. If $\hat{m}=\hat{n}$, then by~\eqref{m0equ} we have
\begin{equation} \label{sumsmall1}
\left\| \sum_{k=M}^N c_k f(k x) \right\| \ll \exp \left( \frac{-\ve \hat{m}^{1-\alpha}}{2 \log \hat{m}}\right).
\end{equation}
If $\hat{m} < \hat{n}$, then by~\eqref{m0equ} and Minkowski's inequality we have
\begin{eqnarray}
& & \left\| \sum_{k=M}^N c_k f(k x) \right\| \nonumber\\
 & \ll & \left\| \sum_{k=M}^{e^{\hat{m}+1}} c_k f(k x) \right\| + \sum_{m=\hat{m}+1}^{\hat{n}-1} \left\| \sum_{k=e^m+1}^{e^{m+1}} c_k f(k x) \right\|  + \left\| \sum_{k=e^{\hat{n}}+1}^{N} c_k f(k x) \right\| \nonumber\\
& \ll & \sum_{m=\hat{m}}^\infty \exp \left( \frac{-\ve m^{1-\alpha}}{2 \log m}\right). \label{sumsmall2}
\end{eqnarray}
Both~\eqref{sumsmall1} and~\eqref{sumsmall2} can be made arbitrarily small if $\hat{m}$ is assumed to be sufficiently large (note that~\eqref{sumsmall2} is the tail of a convergent series). Thus by the Cauchy convergence test the series $\sum_{k=1}^\infty c_k f(kx)$ is convergent in $L^2$. In a similar way, using Lemma~\ref{lemma1b} instead of Lemma~\ref{lemma1}, we obtain for any $M<N$
$$
\left\| \max_{M<L \leq N} \left| \sum_{k=M}^L c_k f(k x) \right| \right\|  \ll \sum_{m=\hat{m}}^\infty \exp \left( \frac{-\ve m^{1-\alpha}}{2 \log m}\right),
$$
where $\hat{m}$ is defined as before. Again the right-hand side can be made arbitrarily small if $M$ is assumed to be sufficiently large. Thus the monotone convergence theorem and Lemma~\ref{lemma4} imply that the series $\sum_{k=1}^\infty c_k f(kx)$ is almost everywhere convergent.
\end{proof}

\begin{proof}[Proof of the optimality of Theorem~\ref{th1}:]
For given $\alpha \in (1/2,1)$, we will show that there exists a sequence $(c_k)_{k \geq 1}$ satisfying~\eqref{th1cond} for a ``small'' value of $K$, for which for the function $f(x)=f_\alpha(x)$ from~\eqref{flambda} the series $\sum_{k=1}^\infty c_k f_\alpha(k x)$ is divergent in $L^2$. We will construct $(c_k)_{k \geq   1}$ such that it is supported on a set of indices which have a small number of prime factors; this idea already appears in~\cite{abs,gal,hilber} and other places. However, there it is only used to construct a finite sequence, whereas in the present case we have to construct an infinite sequence. Note that by~\eqref{mikol},~\eqref{hurw1} and~\eqref{hurw2} the $L^2$ norm of sums of dilated functions $f_\alpha(x)$, $\bar{f}_\alpha(x)$ and $\zeta(1-\alpha,x)$ is the same, up to multiplication with a constant, and consequently we could also use the functions $\bar{f}_\alpha(x)$ or $\zeta(1-\alpha,x)$ instead of $f_\alpha(x)$. \\

We write $(p_r)_{r \geq 1}$ for the sequences of primes in increasing order. We define sets $\Delta_i$ in the following way: for given $i \geq 1$, the set $\Delta_i$ contains those positive integers which are of the form
$$
2^{2i} p_1^{w_1} p_2^{w_2} \dots p_i^{w_i} \qquad \textrm{for} \qquad (w_1, \dots, w_i) \in \{0,1\}^i.
$$
By construction the sets $\Delta_i,~i \geq 1,$ are mutually disjoint (since all numbers in $\Delta_i$ are multiples of either $2^{2i}$ or $2^{2i+1}$, but not of $2^{2i+2}$). Note that the number of elements of $\Delta_i$ is $2^i$.\\

Let $\ve>0$ be fixed, and set $\eta=(1-2\ve)/(1+\ve)$. We define
$$
c_k =  \left\{  \begin{array}{ll} 2^{-i/2} i^{-1} \exp\left( - \frac{\eta}{2(1-\alpha)} (\log k)^{1-\alpha} (\log \log k)^{-1} \right) & \textrm{if $k \in \Delta_i$ for some $i \geq 1$,} \\ 0 & \textrm{otherwise.} \end{array}\right.
$$
Then we have
\begin{eqnarray*}
\sum_{k=1}^\infty c_k^2 \exp\left(\frac{\eta}{1-\alpha} (\log k)^{1-\alpha} (\log \log k)^{-1} \right) & = & \sum_{i=1}^\infty \sum_{k \in \Delta_i} 2^{-i} i^{-2} \\
& = & \sum_{i=1}^\infty i^{-2} \qquad < \infty.
\end{eqnarray*}
By the prime number theorem for all sufficiently large $i$ for all $k \in \Delta_i$ we have
$$
k \leq 2^{2i} \prod_{r=1}^i p_r \leq 2^{2i} \left(((1+\ve) i \log i)^i\right),
$$
and consequently for sufficiently large $i$ and for all $k \in \Delta_i$
$$
\frac{(\log k)^{1-\alpha}}{(\log \log k)} \leq (1+\ve) (i \log i)^{1-\alpha} (\log i)^{-1} = (1+\ve) i^{1-\alpha} (\log i)^{-\alpha}.
$$
Thus for $i \geq 1$ for all $k \in \Delta_i$ we have
\begin{equation} \label{cksizeequ}
c_k \gg 2^{-i/2} i^{-1} \exp\left( - \frac{\eta (1+\ve)}{2(1-\alpha)}  i^{1-\alpha} (\log i)^{-\alpha} \right).
\end{equation}
Using the second part of Lemma~\ref{lemmaalpha} and the facts that $f_\alpha$ has only positive Fourier coefficients and that all coefficients $c_k$ are non-negative, we have
\begin{eqnarray}
\lim_{N \to \infty} \left\| \sum_{k=1}^N c_k f_\alpha(k x) \right\|^2 & \geq & \lim_{M \to \infty} \left\|\sum_{i=1}^M ~\sum_{k \in \Delta_i} c_k f_\alpha(k x) \right\|^2 \nonumber\\
& \geq & \lim_{M \to \infty} \sum_{i=1}^M \left\| \sum_{k \in \Delta_i} c_k f_\alpha(k x) \right\|^2 \nonumber\\
& = & \sum_{i=1}^\infty ~\sum_{k,\ell \in \Delta_i} c_k c_{\ell} \frac{(\gcd(k,\ell))^{2 \alpha}}{(k \ell)^\alpha}. \label{cksizeequ2}
\end{eqnarray}
By the structure of the set $\Delta_i$ for any fixed $k \in \Delta_i$ we have
\begin{equation*}
\sum_{\ell \in \Delta_i} \frac{(\gcd(k,\ell))^{2 \alpha}}{(k \ell)^\alpha} = \prod_{r=1}^i \left(1 + p_r^{-\alpha}\right),
\end{equation*}
which implies that
\begin{equation} \label{struct}
\sum_{k,\ell \in \Delta_i} \frac{(\gcd(k,\ell))^{2 \alpha}}{(k \ell)^\alpha} = 2^i \prod_{r=1}^i \left(1 + p_r^{-\alpha}\right)
\end{equation}
(an argument of this type already appears in G{\'a}l's paper~\cite{gal}). By the prime number theorem we have
$$
\prod_{r=1}^i \left(1 + p_r^{-\alpha}\right) \gg \exp \left(\frac{1-\ve}{1-\alpha} i^{1-\alpha} (\log i)^{-\alpha} \right).
$$
Combining~\eqref{cksizeequ},~\eqref{cksizeequ2} and~\eqref{struct} we get
\begin{eqnarray}
\lim_{N \to \infty} \left\| \sum_{k=1}^N c_k f_\alpha(k x) \right\|^2 & \gg & \sum_{i=1}^\infty i^{-2} \exp\left(\frac{(1-\ve) - \eta  (1+\ve)}{1-\alpha}  i^{1-\alpha} (\log i)^{-\alpha} \right). \label{rhsest2}
\end{eqnarray}
Note that $(1-\ve) - \eta  (1+\ve)=\ve$, and thus the series on the right-hand side of~\eqref{rhsest2} is divergent. Consequently the series $\sum_{k=1}^\infty c_k f_\alpha(kx)$ is divergent in $L^2$, although $(c_k)_{k \geq 1}$ satisfies the extra convergence condition~\eqref{th1cond} for $K=\eta/(1-\alpha)$. Note that by choosing $\ve$ small, $\eta$ can be moved arbitrarily close to 1. This proves the optimality of Theorem~\ref{th1}, apart from the precise optimal value of the constant $K$ in~\eqref{th1cond}.\\
\end{proof}

\begin{proof}[Proof of the convergence part of Theorem~\ref{th2}: ]
The proof of the convergence part of Theorem~\ref{th2} can be given in exactly the same way as the proof of the convergence part of Theorem~\ref{th1} above, using the second part of Lemma~\ref{lemma1} and~\ref{lemma1b} instead of the first part, respectively.\\
\end{proof}

\begin{proof}[Proof of the optimality of Theorem~\ref{th2}: ]
The optimality of condition~\eqref{th2cond} in the case of $L^2$ convergence can be shown in a similar way as the optimality of condition~\eqref{th1cond} in Theorem~\ref{th1}. Again we construct a set of integers which is composed of a relatively small number of prime factors. In particular, again we will use an equality similar to~\eqref{struct}, which allows a precise computation of the corresponding GCD sum. Again we choose $f=f_\alpha$, but as in the proof of the optimality of Theorem~\ref{th1} we could  also use the functions $\bar{f}_\alpha$ or $\zeta(1-\alpha,\cdot)$ instead. The main difference between the present case and the proof of Theorem~\ref{th1} is the fact that we can make the sequence $(n_k)_{k \geq 1}$ grow as fast as we wish. Together with the well-established principle that lacunary sequences of functions show almost independent behavior, this is the reason why for Theorem~\ref{th2} we can also prove optimality with respect to almost everywhere convergence (which was not possible for 
Theorem~\ref{th1}).\\

First we recall that $f_\alpha \in L^{p}(0,1)$ for $p < (1-\alpha)^{-1}$, which was established in~\eqref{falphalp}. Thus we can choose $\delta \in (0,1)$ such that $2+\delta < (1-\alpha)^{-1}$. Furthermore, we can find a number $\beta \in (0,1)$ which satisfies
$$
\beta < \frac{\delta}{2+\delta}.
$$
For this number $\beta$ we have
\begin{equation} \label{betasize}
\left(-\frac{1}{2}+\frac{\beta}{2} \right) (2+\delta) < -1.
\end{equation}
Let $(p_r)_{r \geq 1}$ denote the sequence of primes in increasing order. We set $A(1)=1$ and
$$
A(i) = \left\lceil \beta \log_2 i \right\rceil, \qquad i \geq 2.
$$
We define the numbers $S_i$ and $T_i$ recursively in the following way:
\begin{itemize}
\item $S_1=2,$
\item $T_i = S_i + \left\lceil \log_2 \left( \prod_{r=1}^{A(i)} p_r \right) \right\rceil, \qquad i \geq 1,$
\item $S_{i+1} = T_i + \left\lceil \eta \log_2 i \right\rceil, \qquad  i \geq 1, \qquad \textrm{where $\eta = 12/(2 \alpha - 1)$}$.
\end{itemize}
Then obviously the numbers $(S_i)_{i \geq 1}$ and $(T_i)_{i \geq 1}$ satisfy the conditions of Lemma~\ref{lemmaber}. For $i \geq 1$, we define $\Delta_i$ as the set of all numbers $k$ of the form
$$
k = 2^{S_i} \prod_{r=1}^{A(i)} p_r^{w_r}, \qquad \textrm{where} \qquad (w_1, \dots, w_{A(i)}) \in \{0,1\}^{A(i)}.
$$
Then clearly all elements of $\Delta_i$ are divisible by $2^{S_i}$, and $\Delta_i \subset [2^{S_i},2^{T_i}]$; that is, the sets $\Delta_i$ also satisfy the assumptions of Lemma~\ref{lemmaber}. Let $(n_k)_{k \geq 1}$ denote the sequence consisting of the elements of $\bigcup_{i \geq 1} \Delta_i$, sorted in increasing order. Note that by definition we have
$$
\# \Delta_i=2^{A(i)} \in \left[i^\beta, 2 i^{\beta}\right].
$$
Furthermore we define sets of integers $\Gamma_i,~i \geq 1$, such that
$$
k \in \Gamma_i \qquad \textrm{if and only if} \qquad n_k \in \Delta_i.
$$
Then $(\Gamma_i)_{i \geq 1}$ is a decomposition of $\N$. Let $K_1$ denote a ``small'' constant with a value to be determined later. For every $k \geq 1$ there is an $i$ such that $k \in \Gamma_i$, and we define
$$
c_k = i^{-\beta/2-1/2} (\log i)^{-1} \exp\left( - \frac{K_1 (\log i)^{1-\alpha}}{2(\log \log i)^{\alpha}} \right).
$$
Note that the value of $c_k$ only depends on the index $i$ for which $k \in \Gamma_i$. Thus we can also define numbers $(d_i)_{i \geq 1}$ such that
$$
d_i = c_k \qquad \textrm{whenever} \qquad k \in \Gamma_i, \qquad \textrm{for $i \geq 1, ~k \geq 1$},
$$
which implies that
$$
\sum_{k \in \Delta_i} c_k f_\alpha (k x) = d_i \sum_{k \in \Gamma_i} f_\alpha (n_k x).
$$
Furthermore we have
\begin{eqnarray}
\sum_{k=1}^\infty c_k^2 \exp \left( \frac{K_1 (\log i)^{1-\alpha}}{(\log \log i)^{\alpha}} \right) & = & \sum_{i \geq 1} \sum_{k \in \Gamma_i} i^{-\beta-1} (\log i)^{-2} \cdot \underbrace{\# \Gamma_i}_{\leq 2 i^\beta} \label{lhs1}\\
& \leq & 2 \sum_{i \geq 1} i^{-1} (\log i)^{-2}. \label{lhs2}
\end{eqnarray}
Since the series in~\eqref{lhs2} is convergent, the same holds for the series on the left-hand side of~\eqref{lhs1}. Furthermore, since for $k \in \Gamma_i$ we have
$$
k \ll i^{\beta+1},
$$
the convergence of the left-hand side of~\eqref{lhs1} implies that there exists a positive constant $K_2$ (depending on $K_1$) such that
$$
\sum_{k=1}^\infty c_k^2 \exp \left( K_2 (\log k)^{1-\alpha} (\log \log k)^{-\alpha} \right) < \infty.
$$
As in the lines following~\eqref{struct} we get
\begin{eqnarray}
\sum_{k,\ell \in \Gamma_i} \frac{(\gcd(n_k,n_\ell))^{2 \alpha}}{(n_k n_\ell)^\alpha} & = & \sum_{k,\ell \in \Delta_i} \frac{(\gcd(k,\ell))^{2 \alpha}}{(k \ell)^\alpha} \nonumber\\
& = & \# \Delta_i \prod_{r=1}^{A(i)} (1 + p_r^{-\alpha}) \nonumber\\
& \gg & i^\beta \exp \left(K_3 (\log i)^{1-\alpha} (\log \log i)^{-\alpha} \right) \label{gcdsumdelta}
\end{eqnarray}
for some positive constant $K_3$. Together with the second part of Lemma~\ref{lemmaalpha} this implies that
\begin{equation}\label{l2normsalpha}
\left\| \sum_{k \in \Gamma_i} c_k f_\alpha(n_k x) \right\|^2 \gg i^{-1} (\log i)^{-2} \exp\left((K_3-K_1) (\log i)^{1-\alpha} (\log \log i)^{-\alpha} \right).
\end{equation}
Since all coefficients $(c_k)_{k \geq 1}$ are non-negative we have
\begin{eqnarray*}
\lim_{N \to \infty} \left\| \sum_{k=1}^N c_k f_\alpha(n_k x) \right\|^2 & \geq & \lim_{M \to \infty} \sum_{i=1}^M \left\| \sum_{k \in \Delta_i} c_k f_\alpha(n_k x) \right\|^2.
\end{eqnarray*}
Combining this with~\eqref{l2normsalpha} we arrive at
\begin{eqnarray}
& & \lim_{N \to \infty} \left\| \sum_{k=1}^N c_k f_\alpha(n_k x)\right\|^2 \nonumber\\
& \gg & \lim_{M \to \infty} \sum_{i=1}^M i^{-1} (\log i)^{-2} \exp\left((K_3-K_1) (\log i)^{1-\alpha} (\log \log i)^{-\alpha} \right). \label{l2conv}
\end{eqnarray}
We can assume that $K_1$ was chosen so small that $K_1 < K_3$. Then since the right-hand side of~\eqref{l2conv} is divergent, the series $\sum_{k=1}^\infty c_k f_\alpha(n_k x)$ is divergent in $L^2$. This proves the optimality of Theorem~\ref{th2} for $L^2$ convergence (except for the exact value of the constant $K$ in the extra divergence condition).\\

To show that Theorem~\ref{th2} is also optimal with respect to almost everywhere convergence, we apply Lemma~\ref{lemmaber}. As noted before, Lemma~\ref{lemmaber} can be used for $S_i,~T_i,~\Delta_i$ as defined above. Consequently there exist \emph{independent} random variables $Y_1, Y_2, \dots$ on $((0,1),\mathcal B, {\mathbb P})$ such that
\begin{equation} \label{approx}
\left\| d_i Y_i - \sum_{k \in \Gamma_i} c_k f_\alpha (n_k x) \right\| \ll d_i \left\| Y_i - \sum_{k \in \Delta_i} f_\alpha (k x) \right\| \ll i^{-\beta/2-1/2} i^{-2} \# \Delta_i \ll i^{-5/2+\beta/2}.
\end{equation}
The proof of Lemma~\ref{lemmaber} shows that the random variables $Y_i$ are constructed as the conditional expectation of $\sum_{k \in \Delta_i} f_\alpha (n_k x)$ with respect to some appropriate $\sigma$-fields. Thus the conditional form of Jensen's inequality (see for example~\cite[Theorem~13.3]{ks}) implies that
\begin{equation} \label{lya1}
\E \left(\left|d_i Y_i \right|^{2+\delta}\right) \leq d_i^{2+\delta} \E \left(\left( \sum_{k \in \Gamma_i} f_\alpha (n_k \cdot) \right)^{2+\delta}\right).
\end{equation}
We have chosen $\delta$ in such a way that $f_\alpha \in L^{2+\delta}(0,1)$. Thus by Minkowski's inequality we have
\begin{equation*}
\left\| \sum_{k \in \Gamma_i} f_\alpha (n_k \cdot) \right\|_{2+\delta} \leq \|f_\alpha \|_{2+\delta} \underbrace{\# \Gamma_i}_{=\# \Delta_i} \ll i^{\beta},
\end{equation*}
which together with~\eqref{lya1} implies that
\begin{equation} \label{lya2}
\E \left(\left|d_i Y_i \right|^{2+\delta}\right) \ll i^{(\beta/2-1/2)(2+\delta)}.
\end{equation}
On the other hand, by~\eqref{l2normsalpha} and~\eqref{approx} we have
\begin{equation} \label{lya3}
\E \left((d_i Y_i)^2\right) \gg i^{-1} (\log i)^{-2} \exp\left(K_4 (\log i)^{1-\alpha} (\log \log i)^{-\alpha} \right),
\end{equation}
where $K_4 := K_3-K_1$ is a positive constant (again we assume that $K_1$ was chosen sufficiently small). Let
$$
B_M = \sum_{i=1}^M \E \left((d_i Y_i)^2\right), \qquad \qquad D_M = \sum_{i=1}^M \E \left(\left| d_i Y_i \right|^{2+\delta}\right),
$$
and
$$
F_M(t) = \p \left( x \in (0,1):~\sum_{i=1}^M d_i Y_i < t \sqrt{B_M} \right).
$$
By~\eqref{betasize} and~\eqref{lya2} we see that the sequence $(D_M)_{M \geq 1}$ is bounded. On the other hand, by~\eqref{lya3}, we have
\begin{equation} \label{varianceconv}
B_M \gg \exp\left(K_5 (\log M)^{1-\alpha} (\log \log M)^{-\alpha} \right)
\end{equation}
for some positive constant $K_5$, which in particular implies that $B_M \to \infty$ as $M \to \infty$. Thus, the so-called Lyapunov condition for the central limit theorem is satisfied, which implies that
$$
\sup_{t \in \R} \left| F_M(t) - \Phi(t) \right| \ll L_M \qquad \textrm{as} \qquad M \to \infty,
$$
where
$$
L_M := \frac{D_M}{B_M^{1+\delta/2}}
$$
and $\Phi$ is the standard normal distribution. (For Lyapunov's central limit theorem, see for example \S 1.1 and \S 1.2 of~\cite{petrov}.) Consequently we have
$$
\p \left( \left| \sum_{i=1}^M d_i Y_i \right| \geq  \frac{\sqrt{B_M}}{\log M} \right) \to 1 \qquad \textrm{as $M \to \infty$},
$$
which together with~\eqref{varianceconv} implies that
$$
\limsup_{M \to \infty} \left| \sum_{i=1}^M d_i Y_i \right| = \infty \qquad \textup{a.e.}
$$
Now~\eqref{approx} and the first Borel--Cantelli lemma imply that we also have
$$
\limsup_{M \to \infty} \left| \sum_{i=1}^M \sum_{k \in \Gamma_i} c_k f_\alpha(n_k x) \right| = \infty \qquad \textup{a.e.},
$$
which implies that
$$
\limsup_{N \to \infty} \left| \sum_{k=1}^N c_k f_\alpha(n_k x) \right| = \infty \qquad \textup{a.e.}
$$
This proves the optimality of Theorem~\ref{th2} for almost everywhere convergence.\\

We note that a more detailed analysis shows that a possible choice for the constant $K_1$, and accordingly also for the constant $\hat{K}(\alpha)$ in the statement of Theorem~\ref{th2}, is
$$
K_1=\left((2\alpha-1)/(2\alpha ~\log 2)\right)^{1-\alpha} (1-\alpha)^{-1} - \ve
$$
for an arbitrary $\ve>0$. Consequently the ``blowup'' of the constant in the extra convergence condition is of order $(1-\alpha)^{-1}$ as $\alpha \to 1$, both in the sufficiency condition and in the optimality result.
\end{proof}

\begin{proof}[Proof of Theorem~\ref{th3}:] By~\eqref{normest} and Lemma~\ref{lemmahil} for any real sequence $(c_k)_{k\ge 1}$ and for any $M,N$ satisfying $1 \leq M <  N$ we have
\begin{equation}\label{est}
\int_0^1 \left(\sum_{k=M}^N c_k f(kx)\right)^2 dx \ll \sum_{k, \ell=M}^N |c_k c_\ell|\frac{(\gcd (k, \ell))^{2\alpha}}{(k\ell)^\alpha}
\ll \sum_{k\le N^2} \hat{b}_k^2,
\end{equation}
where the numbers $\hat{b}_k$ are defined by
\begin{equation}\label{b5}
\hat{b}_k=\frac{1}{k^\alpha} \sum_{d|k,~d \geq M} d^\alpha |c_d|.
\end{equation}
Let $\ve>0$ be so small that $1-2\alpha+\ve <0$, and that~\eqref{sigma1} holds. For the simplicity of the formulas, in the sequel we write $\sigma (k)$ for $\sigma_{1-2\alpha+\ve} (k)$. By~\eqref{b5} and the Cauchy--Schwarz inequality we have
\begin{align*}
\hat{b}_k^2 & =\left(\sum_{d|k,~d \geq M} |c_d| (d/k)^\alpha\right)^2 \\
& =\left(\sum_{d|k,~d \geq M} |c_d|
(d/k)^{\frac{1}{2}+\frac{\ve}{2}} (k/d)^{-\alpha+\frac{1}{2}+\frac{\ve}{2}}\right)^2
\\
&\le \sum_{d|k,~d \geq M} c_d^2 (d/k)^{1+\ve} \sum_{d|k,~d \geq M}
(k/d)^{1-2\alpha+\ve} \\
& =\sum_{d|k,~d \geq M} c_d^2 (d/k)^{1+\ve} \sum_{h|k,~h \leq k/M}
h^{1-2\alpha+\ve}\\
& \leq \sum_{d|k,~d \geq M} c_d^2 (d/k)^{1+\ve} \sigma (k).
\end{align*}
Thus
\begin{align}
\sum_{k=M}^{N^2} \hat{b}_k^2 & \le \sum_{k=M}^{N^2}~ \sum_{d|k,~d \geq M} c_d^2
(d/k)^{1+\ve}\sigma(k) \nonumber\\
& \le \sum_{d=M}^{N^2}c_d^2
d^{1+\ve}\sum \sigma (k)k^{-(1+\ve)}, \label{b6}
\end{align}
where the inner sum is extended for all $k$ of the form $k=jd$,
$j=1, 2, \ldots$. But $\sigma (jd)\le \sigma (d)\sigma (j)$ and thus
the inner sum in~\eqref{b6} is bounded by
\begin{equation*}\label{b7}
\sum_{j=1}^\infty \sigma (d) \sigma (j)(dj)^{-(1+\ve)} \ll \sigma (d)d^{-(1+\ve)} \underbrace{\sum_{j=1}^\infty \sigma (j)j^{-(1+\ve)}}_{\ll 1} \ll \sigma (d)d^{-(1+\ve)},
\end{equation*}
where we used the fact that $\sigma (j)\le d(j)=O(j^\eta)$ for any $\eta>0$. Substituting this
into~\eqref{b6}, we get, together with~\eqref{est}, that
\begin{equation} \label{lastf}
\int_0^1 \left(\sum_{k=M}^N c_k
f(kx)\right)^2 dx \ll \sum_{k=M}^{N^2} c_k^2 \sigma(k).
\end{equation}
By~\eqref{sigma1} the right-hand side of~\eqref{lastf} can be made arbitrarily small if $M$ is chosen sufficiently large. Thus by the Cauchy convergence test the series~\eqref{fkx} is convergent in $L^2$.\\

To prove the second part of Theorem~\ref{th3}, let $\alpha\in (1/2, 1)$, $0<\beta<1$, and choose $\delta>0$ so small that $\beta(1+\delta)<1$.
Then by the second statement of Theorem~\ref{th1} there exist a function $f \in C_\alpha$ and a sequence $(c_k)_{k \geq 1}$ such
that
\begin{equation} \label{b8}
\sum_{k=1}^\infty c_k^2 \exp \left( \frac{\beta(1+\delta)}{1-\alpha} \frac{(\log k)^{1-\alpha}}{\log \log k} \right) < \infty
\end{equation}
but the series~\eqref{fkx} does not converge in $L^2$ norm. In view of~\eqref{grw}, the terms of the sum in~\eqref{sigma2} are smaller
than those of~\eqref{b8} for sufficiently large $k$ and thus the sum~\eqref{sigma2} converges,
proving the second half of Theorem~\ref{th3}.
\end{proof}


\end{document}